\documentclass[10pt, leqno]{article}

\usepackage{amsmath,amssymb,amsthm, epsfig}

\usepackage{hyperref}

\usepackage{amsmath}

\usepackage{amssymb}

\usepackage{hyperref}

\usepackage{color}

\usepackage{ulem}

\usepackage{epsfig}

\usepackage{dsfont}

\usepackage{mathrsfs}

\usepackage{wasysym}

\title{Free boundary regularity for a degenerate fully non-linear elliptic problem with right hand side}
\author{\it by \smallskip \\
Raimundo Leit\~ao\footnote{\noindent \textsc{Raimundo Leit\~ao}.
Universidade Federal do Cear\'{a} - UFC. Department of Mathematics. Fortaleza - CE, Brazil - 60455-760.
\texttt{E-mail address: rleitao@mat.ufc.br}
}
\,\,\, $\&$ 
 Gleydson C. Ricarte\footnote{\noindent \textsc{Gleydson Chaves Ricarte}.
Universidade Federal do Cear\'{a} - UFC. Department of Mathematics. Fortaleza - CE, Brazil - 60455-760.
\texttt{E-mail address: ricarte@mat.ufc.br}
}}

\newlength{\hchng}
\newlength{\vchng}
\def \R {\mathbb{R}}

\newtheorem{theorem}{Theorem}[section]

\newtheorem{lemma}[theorem]{Lemma}

\newtheorem{corollary}[theorem]{Corollary}

\theoremstyle{definition}

\newtheorem{definition}[theorem]{Definition}

\theoremstyle{remark}

\newtheorem{remark}[theorem]{Remark}

\numberwithin{equation}{section}

\newcommand{\intav}[1]{\mathchoice {\mathop{\vrule width 6pt height 3 pt depth  -2.5pt
\kern -8pt \intop}\nolimits_{\kern -6pt#1}} {\mathop{\vrule width
5pt height 3  pt depth -2.6pt \kern -6pt \intop}\nolimits_{#1}}
{\mathop{\vrule width 5pt height 3 pt depth -2.6pt \kern -6pt
\intop}\nolimits_{#1}} {\mathop{\vrule width 5pt height 3 pt depth
-2.6pt \kern -6pt \intop}\nolimits_{#1}}}

\begin{document}
\maketitle

\begin{abstract}
We consider an one-phase free boundary problem for a degenerate fully non-linear elliptic operators with non-zero right hand side. We use the approach present in \cite{DeSilva} to prove that flat free boundaries and Lipschitz free boundaries are $C^{1, \gamma}$.
\end{abstract}

\medskip

\noindent{\sc keywords}: free
boundary problems, degenerate fully non-linear elliptic operators, regularity theory.

\smallskip


\tableofcontents

\section{Introduction}

Given a bounded domain $\Omega \subset \mathbb{R}^n$ and $\mu \geq 0$ we consider the degenerate fully non-linear elliptic problem

\begin{eqnarray}
\label{P 5.1. introduc}
	\left \{ 
		\begin{array}{lll}
			\mathfrak{L}_{\mu}u =  f, & \text{ in } & \Omega_{+}\left( u \right),\\
			|\nabla u| = Q, & \text{ on } & \mathfrak{F}(u),
		\end{array}
	\right.
\end{eqnarray}
where $\mathfrak{L}_{\mu}u := \vert \nabla u \vert^{\mu} \Delta u$, $Q \geq 0$ is a $C^{0,\alpha}$-continuous function, $f \in L^{\infty}\left( \Omega \right) \cap C\left(  \Omega \right)$ and
$$
	\Omega^{+}(u):= \{x \in \Omega : u(x) >0\} \quad \textrm{and} \quad \mathfrak{F}(u) := \partial \Omega^{+}(u) \cap \Omega.
$$

The study of the regularity of the free boundary $\mathfrak{F}(u)$, to the problem \eqref{P 5.1. introduc} has a large literature:

\begin{enumerate}
\item \textit{Non-degenerate.}  The case $\mu = f = 0$, was studied in the seminal works of Caffarelli: \cite{AC}, \cite{C1}, \cite{C2}. In the context of fully non-linear elliptic equations, the homogeneuos problem $f=0$ was addressed in \cite{Fe1}, \cite{Fe2}, \cite{W1}, \cite{W2}, \cite{Fel1}, \cite{Fel2} . The non-homogeneous case $f\neq 0$ was studied in  \cite{DeSilva} and \cite{DFS1}. 
\item \textit{Degenerate.} For $\mu > 0$, there are not results about problem \eqref{P 5.1. introduc}.
\end{enumerate} 



In this paper we will develop the regularity theory of $\mathfrak{F}(u)$. Precisely, we will apply the technique presented in \cite{DeSilva} to prove that flat free boundaries are $C^{1, \gamma}$ (see section 2 for the definition of viscosity solutions): 

\begin{theorem} 
\label{th 5.1. introd.}
Let $u$ be a viscosity solution to (\ref{P 5.1. introduc}) in ball $B_{1}\left( 0\right)$. Suppose that $0 \in \mathfrak{F}\left( u \right) $ and $Q\left( 0 \right)= 1$. There exists a universal constant $\tilde{\varepsilon} > 0$ such that, if the graph of $u$ is $\tilde{\varepsilon}$-flat in $B_{1}\left( 0\right)$, i.e.  
\begin{eqnarray*}
\label{t 5.1.1. introd.}
\left( x_{n} - \tilde{\varepsilon}\right)^{+} \leq u \left( x \right) \leq \left( x_{n} + \tilde{\varepsilon} \right)^{+} \ \ \mbox{for} \ \ x \in B_{1}\left( 0 \right),
\end{eqnarray*} 
and
\begin{eqnarray*}
\label{t 5.1.2. introd.}
\Vert f \Vert_{L^{\infty}\left( B_{1}\left( 0\right) \right) } \leq \tilde{\varepsilon}, \ \ \left[ Q \right]_{C^{0,\alpha}\left( B_{1}\left( 0\right) \right)} \leq \tilde{\varepsilon},
\end{eqnarray*}
then $\mathfrak{F}\left( u \right)$ is $C^{1,\beta}$ in $B_{\frac{1}{2}}\left( 0 \right)$.
\end{theorem}

As in \cite{DeSilva}, the strategy of the proof of Theorem \ref{th 5.1. introd.} is to obtain the \textbf{improvement of flatness} property for the graph of a solution $u$: if the graph of $u$ oscillates away $\varepsilon$ from a hyperplane in $B_{1}$ then in $B_{\delta_{0}}$ it oscillates $\frac{\delta_{0}\varepsilon}{2}$ away from possibly a different hyperplane. The fundamental steps to achieve this propery are:  Harnack type Inequality and Limiting solution.  In our problem, the structure of the operator $\mathfrak{L}_{\mu}$ requires some changes. In next section, we comment on the main difficulties we came across and how to overcome them. 

Moreover, through a blow-up from Theorem \ref{th 5.1. introd.} and the approach used in \cite{C1}, we obtain the our second main result:

\begin{theorem}[Lipschitz implies $C^{1,\beta}$] \label{Holder1}
Let $u$ be a viscosity solution for the free boundary problem
\begin{eqnarray*}
	\left \{ 
		\begin{array}{lll}
			\mathfrak{L}_{\mu}u  =  f, & \text{ in } & \Omega_{+}\left( u \right),\\
			|\nabla u| = Q, & \text{ on } & \mathfrak{F}(u). 
		\end{array}
	\right.
\end{eqnarray*}
 Assume that $0 \in \mathfrak{F}(u)$, $f \in L^{\infty}(B_1)$ is continuous in $B^{+}_{1}(u)$ and $Q(0)>0$. 
If $\mathfrak{F}(u)$ is a Lipschitz graph in a neighborhood of $0$, then $\mathfrak{F}(u)$ is $C^{1,\beta}$ in a (smaller) neighborhood of $0$.
\end{theorem}

In Theorem \ref{Holder1}, the size of the neighborhood where $\mathfrak{F}(u)$ is $C^{1,\beta}$ depends on the radius $r$ of the ball $B_r$ where $\mathfrak{F}(u)$ is Lipschitz, on the Lipschitz norm of $\mathfrak{F}(u)$, on n, $\alpha$ and $\|f\|_{\infty}$. We also emphasize that to obtain the Theorem \ref{Holder1} via the \textit{improvement of flatness propery} for the graph of $u$, we will need Lipschitz regularity and non-degeneracy for $u$. As in \cite{DeSilva}, we will use \textit{Harnack Inequality} and \textit{Maximum Principle} for solutions of the equation $\mathfrak{L}_{\mu} v = f$ in balls to establish Lipschitz regularity and non-degeneracy for $u$. Since we do not have Harnack Inequality available for $n> 2$, see \cite{BD} for $n=2$, the Theorem \ref{Holder1} will be proved for the case $n=2$.

Finally, we believe that Theorems \ref{th 5.1. introd.} and \ref{Holder1} can be
established to the more general operator $\mathfrak{L}_{\mu}u = \vert \nabla u \vert^{\gamma} F(D^{2}u)$, where $F$ is uniformly elliptic and satisfies homogeneity property:

\begin{enumerate}
\item (Ellipticity condition) There exist constants $0 < \lambda \le \Lambda$ such that for any $M,N \in \textrm{Sym}(n)$, with $M \ge 0$ there holds
\begin{equation*} \label{unif1}
	\lambda \|M\| \le F(N + M) - F(N) \le \Lambda \|M\|.
\end{equation*}

\item (Homogeneity condition) For all $t \in \R- \left\lbrace 0 \right\rbrace$ and $M \in \textrm{Sym}(n)$,
$$
F( t M) = | t | F(M).
$$
\end{enumerate}

The paper is organized as follows. In Section 2 we define the notion of viscosity solution to the
free boundary problem \eqref{P 5.1. introduc} and 
gather few tools that we shall use in the proofs of Theorem
\ref{th 5.1. introd.} and Theorem \ref{Holder1}. In
Section 3 we present the proof of Harnack type inequality. Section 4 is devoted to
the proof of improvement of flatness and in Section 5 we establish the regularity of the free boundary $\mathfrak{F}(u)$.

\section{Preliminaries} \label{section2}

\subsection{Notation and Definitions} \label{subsection2}
Let us move towards the hypotheses, set-up and main notations used in this article. For $B_1$ we denote the open unit ball in the Euclidean space $\mathds{R}^n$. Furthermore, if $x\in \mathbb{R}^{n}$ we denote $x=(x_{1}, \dots, x_{n})$. We start by gathering some basic information of the limiting configuration. We shall use viscosity solution setting to access the free boundary regularity theory.
\begin{definition}
Given two continuous functions $u$ and $\phi$ defined in an open $\Omega$ and a point $x_0 \in \Omega$, we say that $\phi$ touches $u$ by below (resp. above) at $x_0$ whenever
$
    u(x_0)=\phi(x_0)
$
$$
    u(x) \ge \phi(x) \,\, (\textrm{resp.} \, \, u(x) \le \phi(x)) \,\,\,\textrm{in a neighborhood} \,\,\mathcal{O} \,\, \textrm{of} \,\, x_0.
$$
If this inequality is strict in $\mathcal{O} \setminus \{x_0\}$, we say that $\phi$ touches $u$ strictly by below (resp. above).
 \end{definition}

\begin{definition}\label{d 5.3}
Let $u \in C(\Omega)$ nonnegative. We say that $u$ is a viscosity solution to
\begin{eqnarray}
 \label{defviscsol}
	\left \{ 
		\begin{array}{lll}
			\mathfrak{L}_{\mu} u  =  f, & \text{ in } & \Omega_{+}\left( u \right),\\
			|\nabla u| = Q, & \text{ on } & \mathfrak{F}(u). 
		\end{array}
	\right.
\end{eqnarray}
 if and only if the following conditions are satisfied:
\begin{enumerate}
\item[(F1)] If $\phi \in C^{2}(\Omega^{+}(u))$ touches $u$ by below (resp. above) at $x_0 \in \Omega^{+}(u)$ then
$$
   \mathfrak{L}_{\mu} \phi(x_0) \le f(x_0) \,\, \left(\textrm{resp.} \,\, \mathfrak{L}_{\mu} \phi(x_0) \ge f(x_0)\right).
$$
\item[(F2)] If $\phi \in C^2(\Omega)$ and $\phi^{+}$ touches $u$ below (resp. above) at $x_0 \in \mathfrak{F}(u)$ and $|\nabla \phi|(x_0) \not= 0$ then
$$
    |\nabla \phi|(x_0) \le Q(x_0) \,\,\, \left( \textrm{resp.} \,\,\, |\nabla \phi|(x_0) \ge Q(x_0)\right).
$$
\end{enumerate}
\end{definition}

We refer to the usual definition of subsolution, supersolution and solution of a degenerate PDE. Let us introduce the notion of comparison subsolution/supersolution.
\begin{definition}\label{d 5.3e}
We say $u \in C(\Omega)$ is a strict (comparison) subsolution (resp. supersolution) to \eqref{P 5.1. introduc} in $\Omega$, if only if $u \in C^{2}(\overline{\Omega^{+}(v)})$ and the following conditions are satisfied:
\end{definition}
 \begin{enumerate}
\item[(G1)] $\mathfrak{L}_{\mu} u > f(x) \,\,\, \left(\textrm{resp.} \,\, < f \right)$ in $\Omega^{+}(u)$;
\item[(G2)] If $x_0 \in \mathfrak{F}(u)$, then
$$
    |\nabla u|(x_0) > Q(x_0) \,\,\, \left(\textrm{resp.} \,\,\, 0 < |\nabla u|(x_0) < Q(x_0)\right).
$$
\end{enumerate}

 Next lemma provides a basic comparison principle for solutions to the free boundary problem \eqref{P 5.1. introduc}. The Lemma below yields the crucial tool in the proof of main result. 
 \begin{lemma}\label{l 5.1} The following remark is an consequence of the definitions above: Let $u,v$ be respectively a solution and a strict subsolution to \eqref{P 5.1. introduc} in $\Omega$. If $u \ge v^{+}$ in $\Omega$ then $u > v^{+}$ in $\Omega^{+}(v) \cup \mathfrak{F}(v)$.
 \end{lemma}
 
As in \cite{DeSilva}, another fundamental tool in the proof of Theorem \eqref{th 5.1. introd.} is the regularity of solutions to the classical Neumann problem for the constant coefficient linear equation
\begin{equation} \label{const}
\left\{
\begin{array}{rcl}
 \Delta u_{\infty} &=& 0,  \quad \mbox{in}  \quad B^{+}_{\rho} ,\\
\frac{\partial u_{\infty}}{\partial \nu}&=& 0,  \quad \mbox{on} \quad \Upsilon_{\rho},
\end{array}
\right.
\end{equation}
where $\nu := (0,0, \ldots,0,1)$ and we denote by
\begin{eqnarray}
	B^{+}_{\rho} &:=& \{x \in \mathbb{R}^n : |x| <\rho, x_n >0\} \label{neumann1} \\
          \Upsilon_{\rho} &:=& \{x \in \mathbb{R}^n : |x| <\rho, x_n=0\}  \label{neumann2}.
\end{eqnarray}
We use the notion of viscosity solution to \eqref{const}:

\begin{definition}\label{]defVisc}
Let $u_{\infty} \in C(B_{\rho} \cap \{x_n \ge 0\})$. We say that $u_{\infty}$ is a viscosity solution to \eqref{const} if given $P(x)$ a quadratic polynomial touching $u_{\infty}$ by below (resp. above) at $x_0 \in B_{\rho} \cap \{x_n \ge 0\}$, then
\begin{enumerate}
\item[(i)] if $ x_0 \in B^{+}_{\rho} $ then $\Delta P(x_0) \le 0$ (resp. $\Delta P(x_0) \ge 0$);
\item[(ii)] if $ x_0 \in \Upsilon_{\rho}$ then $\frac{\partial P(x_0)}{\partial \nu} \le 0$ (resp. $\frac{\partial P(x_0)}{\partial \nu} \ge 0$)
\end{enumerate}
\end{definition}

\begin{remark}\label{remark}
Notice that, in the definition above we can choose polynomials $P$ that touch $u_{\infty}$ strictly by below/above. Also, it suffices to verify that (ii) holds for polynomials $\tilde{P}$ with $\Delta \tilde{P} >0$ (see \cite{DeSilva}). 
\end{remark}

The proof of $C^{2}$-regularity of solutions to the classical Neumann problem is classical and will be omitted (see for example \cite{DeSilva}).
\begin{lemma}\label{reg_Laplace}
Let $u_{\infty}$ be a viscosity solution to
\begin{equation} \label{reg1}
\left\{
\begin{array}{rcl}
\Delta u_{\infty} &=& 0 , \quad \mbox{in}  \quad B^{+}_{\rho} \\
\frac{\partial u_{\infty}}{\partial \nu}&=& 0,  \quad \mbox{on}  \quad \Upsilon_{\rho}
\end{array}
\right.
\end{equation}
with $\|u_{\infty}\|_{L^{\infty}} \le 1$. There exists a universal constant $C_0 >0$ such that 
$$
    |u_{\infty}(x)-u_{\infty}(0) - \nabla u_{\infty}(0) \cdot x| \le C_0 \rho^{2} \quad \textrm{in} \,\, B_{\rho} \cap \{x_n \ge 0\}.
$$\end{lemma}


\subsection{Difficulties and Changes}

In this section, we comment on the main difficulties we came across to obtain the \textit{improvement of flatness} property for the graph of a solution $u$ of \eqref{P 5.1. introduc} and how to overcome them. 

\begin{enumerate}
\item \textbf{Harnack type Inequality}. When we consider the problem \eqref{P 5.1. introduc} for $\mu > 0$, the first difficulty we find lies in the following fact: in general, if $p$ is an affine function and $u$ is a solution to the problem 
\begin{eqnarray}
\label{Degen. prob.}
 |\nabla v|^{\mu}\Delta v  =  f, \quad \text{ in } B_{r}(x_{0}),
\end{eqnarray} 
we can not conclude that $u + p$ is a solution to the equation \eqref{Degen. prob.}. For $\mu = 0$ we know $u + p$ is still solution for \eqref{Degen. prob.}. In \cite{DeSilva}, this fact is important because it allows us to apply Harnack Inequality for $v(x) = u(x) - x_{n}$ which is crucial to reach an \textit{improvement of flatness} for the graph of $u$ . We overcome this difficulty as follows: \\

Step 1. We notice that the function $v(x) = u(x) - x_{n}$ is a solution to the problem
\begin{eqnarray}
\label{Degen. with e}
 \vert \nabla v + e  \vert^{\mu} \Delta v  =  f, \quad \text{ in } B_{r}(x_{0}),
\end{eqnarray} 
where $e \in \mathbb{R}^{n}$ with $|e|=1$. Then, we know from \cite{Imbert Harnack} that $v$ satisfies the following Harnack Inequality 
\begin{eqnarray}
\label{Imbert Harnack 1}
    \sup_{B_{r/2}(x_{0})}v \leq C\left\lbrace \inf_{B_{r/2}(x_{0})} v  +  \max(2,  \Vert f \Vert_{\infty})  \right\rbrace,
\end{eqnarray}
where $C>0$ is a constant depending only on $n$ and $\mu$. \\

Step 2. Since we will use a blow-up argument to prove our main results (Theorem \eqref{th 5.1. introd.} and Theorem \eqref{Holder1}), we can assume that $\Vert f \Vert_{\infty} $ is small. Using the homogeneity property of $\Delta$ we consider the scaling function $v_{r}(x) = \frac{v(rx + x_{0})}{r}$ and apply \eqref{Imbert Harnack 1} to obtain 
\begin{eqnarray}
\label{Scaling Imbert Harnack}
   \sup_{B_{r/2}(x_{0})}v \leq C\left\lbrace \inf_{B_{r/2}(x_{0})} v  + 2r  \right\rbrace.
\end{eqnarray}

Precisely, we use the following result:

\begin{lemma}
\label{delta - Harnack inequality}
Let $u$ be a non-negative viscosity solution to   
\begin{eqnarray}
\label{equatio with e_{n}}
			 \vert \nabla v + e \vert^{\mu} \Delta v =  f, \quad  \text{ in }  B_{\delta},
\end{eqnarray}
where $0< \delta < 1$, $\Vert f \Vert_{\infty} \leq 1$ and $\vert e \vert = 1$. Then, there exists a constant $C$
depending only on $n$ and $\mu$ such that
$$
    \sup_{B_{\delta/ 2}}v \leq C\left\lbrace \left( \inf_{B_{\delta/ 2}} v  + 2\delta \right ) \right\rbrace.
$$
\end{lemma}
\begin{proof}
Define
\begin{eqnarray}
u(x) = \dfrac{v(\delta x)}{\delta}
\end{eqnarray}
for all $ x  \in B_{1}$. Notice that $v$ is a solution to 
\begin{eqnarray}
\label{equatio with e_{n} proof}
			 \vert \nabla u + e  \vert^{\mu} \Delta u =  \delta f (\delta x), \quad  \text{ in }  B_{1}.
\end{eqnarray}
Notice that, if $F(M):= \textrm{Trace} \ M$, the equation can be written as $G(Du,D^2 u)=f$ with
$$
	G(\vec q, M) \colon= |\vec e + \vec q|^{\mu} F(M).
$$
In particular, if $|\vec q| \ge 2$ then $|\vec e + \vec q| \ge 1$ and
$$
\left\{
\begin{array}{ccccccc}
G(\vec q, M) &=& 0\\
|\vec q| &\ge& 2\\
\end{array}
\right.
\Rightarrow
\left\{
\begin{array}{ccccccc}
\mathscr{M}^{+}(D^2 u) + |f| &\ge& 0\\
\mathscr{M}^{-}(D^2 u) - |f| &\le& 0\\
\end{array}
\right.
$$
Thus, from \cite{Imbert Harnack}, we can apply Harnack inequality to obtain
\begin{eqnarray}
\label{for Harnack Inequality 1}
\sup_{B_{1/ 2}}u & \leq & C\left\lbrace \left( \inf_{B_{1/ 2}} u  + \max(2, \delta \Vert f \Vert_{\infty}) \right ) \right\rbrace \\ \nonumber
& \leq & C\left\lbrace \left( \inf_{B_{1/ 2}} u  + 2 \right ) \right\rbrace,
\end{eqnarray}
where $C= C(n, \mu)$ is a positive constant. The Lemma \ref{delta - Harnack inequality} is concluded.
\end{proof}

Step 3. The Harnack Inequality \eqref{Scaling Imbert Harnack} is different from the Harnack Inequality used in \cite{DeSilva}. In fact, for $0< \varepsilon < 1$, DeSilva used the inequality  
\begin{eqnarray}
\label{DeSilva Harnack}
   \sup_{B_{r/2}(x_{0})}v \leq C\left\lbrace \inf_{B_{r/2}(x_{0})} v  + \Vert f \Vert_{\infty}  \right\rbrace
\end{eqnarray}
to prove that if $\Vert f \Vert_{\infty}$ satisfies the \textit{smallness} condition $\Vert f \Vert_{\infty} \leq \varepsilon^{2}$ we can build radial barries $w_{r, x_{0}}$
and apply comparison techniques to achieve an appropriate Harnack type Inequality (see Theorem 3.1 and Lemma 3.3 in \cite{DeSilva}) to establish the \textit{improvement of flatness}. A carefully analysis the behavior of $v = u - x_{n}$ (or $v = x_{n} - u$) in a ball $B_{r_{1}}(x_{0})$ with
$$\vert \nabla u \vert < \frac{1}{2} \quad \text{in} \ B_{r_{1}}(x_{0}), $$ 
and $r_{1} = r_{1}(n, \mu) >0$, reveals that if we consider radial barries $w_{r, x_{0} - r_{2}e_{n}}$ (or $w_{r, x_{0} + r_{2}e_{n}}$) the condition $\Vert f \Vert_{\infty} \leq \varepsilon^{2}$ used in \eqref{DeSilva Harnack} can be replaced by an adequate \textit{smallness} condition of the \textit{radius} $r =r(r_{2})$ in \eqref{Scaling Imbert Harnack} to obtain a Harnack type Inequality, where $r_{2} = r_{2}(r_{1})$.

\item \textbf{Limiting solution}. In the more general case $\mathfrak{L}_{\mu}u = \vert \nabla u \vert^{\gamma} F(D^{2}u)$, where $F$ is uniformly elliptic and satisfies homogeneity property, our \textit{Limiting solution} is given by a classical Neumann problem for the constant coefficient linear equation
\begin{equation} 
\label{General limiting solution}
\left\{
\begin{array}{rcl}
 \mathcal{F}_{0}(D^2 u_{\infty}) &=& 0,  \quad \mbox{in} \,\, B^{+}_{r} ,\\
\frac{\partial u_{\infty}}{\partial \mu}&=& 0,  \quad \mbox{in} \,\, \Upsilon_{r},
\end{array}
\right.
\end{equation}
where $\mu := (0,0, \ldots,0,1)$.  In \cite{Sil},  Emmanouil Milakin and Luis E. Silvestre studied the regularity for viscosity solutions of fully nonlinear uniformly elliptic second order equations with Neumann boundary data.  More precisely,  they showed that viscosity solutions of the homogeneous problem with Neumann boundary data \eqref{General limiting solution} are class $C^{1,\alpha_0}(\overline{B^{+}_{\rho}})$ for some $\alpha \in (0,1)$. We point out that the regularity $C^{1,\alpha_0}(\overline{B^{+}_{\rho}})$ for $u_{\infty}$ is sufficient to obtain the \textit{improvement of flatness}.

\end{enumerate}

\section{Harnack Type Inequality}

In this section, based on comparison principle granted in Lemma \ref{l 5.1}, we prove a Harnack type inequality for a solution $u$ to the problem \eqref{P 5.1. introduc} with  the following conditions:
\begin{eqnarray}
\label{t 5.2.b} \Vert  f \Vert_{L^{\infty}\left( \Omega \right)} \leq \varepsilon^{2}, \\ 
\label{t 5.2.d} \Vert Q - 1 \Vert_{L^{\infty}\left(\Omega \right)} \leq \varepsilon^{2},
\end{eqnarray}
for $0 < \varepsilon < 1$.
\begin{lemma}
\label{Harnack lemma 5.2}
Let $u$ be a viscosity solution to (\ref{P 5.1. introduc}) in $\Omega$, under assumptions  (\ref{t 5.2.b})--(\ref{t 5.2.d}). There exist a universal constant $\tilde{\varepsilon} > 0$ such that if  $0 < \varepsilon \leq \tilde{\varepsilon}$ and $u$ satisfies
\begin{eqnarray}
\label{l 5.2.1}
\\
p^{+}\left( x \right) \leq u \left( x \right) \leq \nonumber \left( p\left( x \right) + \varepsilon \right)^{+}, \ \ \vert \sigma \vert < \frac{1}{20} \ \ \mbox{in} \ \ B_{1}\left( 0 \right), \ \ p\left( x \right) = x_{n} + \sigma, 
\end{eqnarray} 
then if at $x_{0}= \frac{1}{10}e_{n}$ 
\begin{eqnarray}
\label{l 5.2.2}
u \left( x_{0} \right) \geq \left( p\left( x_{0} \right) + \frac{\varepsilon}{2} \right)^{+}, 
\end{eqnarray} 
then
\begin{eqnarray}
\label{l 5.2.3}
u \geq \left( p + c\varepsilon \right)^{+} \ \ \mbox{in} \ \ \overline{B}_{\frac{1}{2}}\left( 0 \right), 
\end{eqnarray}
for some $0 < c < 1$. Analogously, if
\begin{eqnarray}
\label{l 5.2.4}
u \left( x_{0} \right) \leq \left( p\left( x_{0} \right) + \frac{\varepsilon}{2} \right)^{+}, 
\end{eqnarray} 
then
\begin{eqnarray}
\label{l 5.2.5}
u \leq \left( p + \left( 1 - c \right)\varepsilon \right)^{+} \ \ \mbox{in} \ \ \overline{B}_{\frac{1}{2}}\left( 0 \right). 
\end{eqnarray}
\end{lemma}

\begin{proof}
We verify \eqref{l 5.2.3}. The proof of \eqref{l 5.2.5} is analogous. Notice that 
\begin{eqnarray}
\label{l 5.2.11}
B_{\frac{1}{20}}\left( x_{0} \right) \subset B^{+}_{1}\left( u \right).
\end{eqnarray}
From \cite{IS} we know that $u$ is $C^{1, \alpha}$ in $B_{\frac{1}{40}}\left( x_{0} \right)$ with
\begin{eqnarray*}
\left[ u \right]_{1 + \alpha, B_{1/40}(x_{0})} & \leq & C \left( \Vert u \Vert_{\infty} + \Vert f \Vert^{\frac{1}{\mu +1}}_{\infty} \right) \\ \nonumber
& \leq & 3C,
\end{eqnarray*}
where $\alpha = \alpha (n, \mu) \in (0, 1)$ and $C=C(n, \mu)>1$. Now we consider two cases:\\
\item[Case 1]: $\vert \nabla u (x_{0}) \vert < \frac{1}{4}$.\\

Choose $r_{1} = r_{1}(n, \mu) > 0$ such that 

\begin{eqnarray}
\label{NEW 1}
\vert \nabla u \vert \leq \dfrac{1}{2} \quad \text{in} \ B_{r_{1}}(x_{0}).
\end{eqnarray}

There exists a constant $r_{2} = r_{2}(r_{1}) = r_{2}(n, \mu) > 0$ that satisfies
\begin{eqnarray*}
\label{NEW 2}
(x - r_{2} e_{n}) \in B_{r_{1}} (x_{0}), \quad \text{for all} \ x \in  B_{\frac{r_{1}}{2}}(x_{0}).
\end{eqnarray*}
For $r_{3} = \min\left\lbrace \frac{r_{1}}{4}, \frac{r_{2}}{8} \right\rbrace$ we apply the Lemma \ref{delta - Harnack inequality} in $B_{2r_{3}}(x_{0})$ and we obtain

\begin{eqnarray}
\label{NEW 3}
u\left( x \right) - p\left( x \right)  \geq  c_{0}(u\left( x_{0} \right) - p\left( x_{0} \right)) - 4r_{3} \geq \frac{c_{0}\varepsilon}{2} - 4r_{3}
\end{eqnarray}
for all $x \in B_{r_{3}}(x_{0})$. From \eqref{NEW 1} and \eqref{NEW 3} we can write
\begin{eqnarray*}
\label{NEW 4}
\frac{c_{0}\varepsilon}{2} - 4r_{3} & \leq & u\left( x \right) - p\left( x \right) \\ \nonumber 
& = & u((x -r_{2}e_{n}) + r_{2}e_{n}) -p ((x -r_{2}e_{n}) + r_{2}e_{n}) \\ \nonumber 
& = & u((x -r_{2}e_{n}) + r_{2}e_{n}) - p ((x -r_{2}e_{n})) - r_{2} \\ \nonumber 
& \leq &  u ((x -r_{2}e_{n})) - p ((x -r_{2}e_{n})) + \dfrac{r_{2}}{2} - r_{2},
\end{eqnarray*}
for all $x \in B_{r_{3}}(x_{0})$. Thus, we find
\begin{eqnarray}
\label{NEW 5}
\frac{c_{0}\varepsilon}{2} & \leq & \frac{c_{0}\varepsilon}{2} - 4r_{3} +   \dfrac{r_{2}}{2} \\ \nonumber 
& \leq &  u (x) - p(x),
\end{eqnarray} 
for all $x \in B_{r_{3}}(\overline{x}_{0})$, where $\overline{x}_{0} = x_{0} - r_{2}e_{n}$.

 Let $w \colon \overline{D} \rightarrow \mathbb{R}$ be defined by
\begin{eqnarray}
\label{l 5.2.6}
	w\left( x \right) = c\left( \vert x - \overline{x}_{0} \vert^{-\gamma} - \left( \frac{4}{5}\right)^{-\gamma} \right),
\end{eqnarray}
where $D:= B_{\frac{4}{5}}\left( \overline{x}_{0} \right)\setminus \overline{B}_{r_{3}}\left( \overline{x}_{0} \right)$. We choose $c=c(n, \mu) > 0$ such that 
\begin{eqnarray}
\label{l 5.2.7}
w =	\left \{ 
\begin{array}{lll}
0, & \text{on} & \partial B_{\frac{4}{5}}\left( \overline{x}_{0} \right), \\
1, & \text{on} & \partial B_{r_{3}}\left( \overline{x}_{0}\right). 
\end{array}
\right.
\end{eqnarray} Now define
\begin{eqnarray}
\label{l 5.2.15}
v\left( x \right) = p\left( x \right) + \frac{c_{0}\varepsilon}{2} \left( w\left( x \right) - 1 \right), \ \ x \in \overline{B}_{\frac{4}{5}}\left( \overline{x}_{0} \right),
\end{eqnarray}
and for $t \geq 0$,
\begin{eqnarray}
\label{l 5.2.16}
v_{t}\left( x \right) = v\left( x \right) + t, \ \ x \in \overline{B}_{\frac{4}{5}}\left( \overline{x}_{0} \right).
\end{eqnarray}
By choice of $c$ we have $w \leq 1$ in $D$. Then, extending $w$ to $1$ in $B_{r_{3}}\left( \overline{x}_{0} \right)$ we find
\begin{eqnarray}
\label{l 5.2.17}
v_{0}\left( x \right) = v\left( x \right) \leq p\left( x \right) \leq u \left( x \right), \ \ x \in \overline{B}_{\frac{4}{5}}\left( \overline{x}_{0} \right).
\end{eqnarray}
Consider 
$$
	t_{0}= \sup \left\lbrace t \geq 0 : v_{t} \leq u \ \ \mbox{in} \ \ \overline{B}_{\frac{4}{5}}\left( \overline{x}_{0} \right) \right\rbrace.
$$ 
Assume, for the moment, that we have already verified $t_{0} \geq \frac{c_{0}\varepsilon}{2}$. From definition of $v$ we have
\begin{eqnarray*}
\label{}
u\left( x \right) \geq v \left( x \right) + t_{0} \geq p\left( x \right) + \frac{c_{0}\varepsilon}{2} w\left( x \right), \ \ \forall x \in B_{\frac{4}{5}}\left( \overline{x}_{0} \right).    
\end{eqnarray*}
Notice that $ B_{\frac{1}{2}}\left( 0\right) \subset  B_{\frac{3}{5}}\left( \overline{x}_{0} \right)$ and
\begin{eqnarray*}
\label{}
w\left( x \right)  \geq	\left \{ 
\begin{array}{lll}
c \left [ \left( \frac{3}{5}\right)^{-\gamma} - \left( \frac{4}{5}\right)^{-\gamma} \right ] , & \text{in} & B_{\frac{3}{5}}\left( \overline{x}_{0} \right) \setminus B_{r_{3}}\left( \overline{x}_{0} \right), \\
1, & \text{on} &  B_{r_{3}}\left( \overline{x}_{0} \right). 
\end{array}
\right.  
\end{eqnarray*}
Hence, we conclude ($\varepsilon$ small) that
\begin{eqnarray*}
\label{}
u\left( x \right) - p\left( x \right) & \geq & \nonumber c_{1} \varepsilon, \ \ \mbox{in} \ \ B_{1/2}\left( 0 \right),
\end{eqnarray*}
and the result is proved. Let us now prove that indeed $t_{0} \geq \frac{c_{0}\varepsilon}{2}$. For that, we suppose for the sake of contradiction that $t_{0} < \frac{c_{0}\varepsilon}{2}$. Then there would exist $y_{0} \in \overline{B}_{\frac{4}{5}}\left( \overline{x}_{0} \right)$ such that
\begin{eqnarray*}
\label{}
v_{t}\left( y_{0} \right) = u\left( y_{0} \right).
\end{eqnarray*}  
In the sequel, we show that $y_{0} \in B_{r_{3}}\left( \overline{x}_{0}\right)$. From definition of $v_{t}$ and by the fact that $w$ has zero boundary data on $\partial B_{4/5}\left( \overline{x}_{0} \right)$ we have
\begin{eqnarray*}
\label{}
 	v_{t}= p - \frac{c_{0}\varepsilon}{2} + t_{0} < u \ \ \mbox{in} \ \ \partial B_{4/5}\left( \overline{x}_{0} \right),
\end{eqnarray*}
where we have used that $u \geq p$ and $ t_{0} < \frac{c_{0}\varepsilon}{2}$. We compute directly,
\begin{eqnarray}
\label{l 5.2.8}
\\
\partial_{i}w = \nonumber -\gamma \left( x_{i} - \overline{x}^{i}_{0} \right)\vert x - \overline{x}_{0} \vert^{-\gamma - 2} 
\end{eqnarray}
and
\begin{eqnarray}
\label{l 5.2.9}
\\
\partial_{ij}w = \nonumber \gamma \vert x - \overline{x}_{0} \vert^{-\gamma - 2} \left\lbrace \left( \gamma + 2 \right) \left( x_{i} -  \overline{x}^{i}_{0} \right)\left( x_{j} - \overline{x}^{j}_{0} \right)\vert x - \overline{x}_{0} \vert^{- 2} - \delta_{ij} \right\rbrace.
\end{eqnarray}
Moreover, there exists $C=C(n, \mu, \gamma) > 1$ such that $\vert \nabla w \vert \leq C$ in $D$. Then, if $ \varepsilon > 0$ is small we have in $D$
\begin{eqnarray}
\label{NEW 7}
\vert e_{n} + (c_{0}/2) \varepsilon \nabla w \vert^{\mu} \geq \left( \dfrac{1}{2} \right)^{\mu}. 
\end{eqnarray} 
Thus, if $\gamma = \gamma (n) > 1$ is large, from \eqref{l 5.2.9} and \eqref{NEW 7} we find
\begin{eqnarray*}
\label{}
\vert \nabla v_{t} \vert^{\kappa} \Delta v_{t} & \ge & \gamma \varepsilon c_{} \left( \dfrac{1}{2} \right)^{\mu} (5/4)^{\gamma + 2} \left\lbrace \left( \gamma + 2 \right) - n  \right\rbrace \\ \nonumber
& \geq & \delta_{0} \varepsilon  \\ \nonumber
& \geq & \varepsilon^{2}\\ \nonumber
&\geq & f,
\end{eqnarray*} 
where $\delta_{0} = \delta_{0} (n, \mu) >0$. On the other hand, we have
\begin{eqnarray}
\label{t 5.2.18}
\vert \nabla v_{t_{0}} \vert \geq \vert \partial_{n}v \vert = \vert 1 + (c_{0}/2) \varepsilon \partial_{n}w \vert, \ \  \mbox{in} \ \  D.
\end{eqnarray} 
By radial symmetry of $w$, we have
\begin{eqnarray}
\label{t 5.2.19}
\partial_{n}w\left( x \right)  = \vert \nabla w\left( x \right) \vert \langle \nu_{x}, e_{n}\rangle , \ \ x \in D,
\end{eqnarray}  
where $\nu_{x}$ is the unit vector in the direction of $x - \overline{x}_{0}$. From (\ref{l 5.2.8}) we have
\begin{eqnarray*}
\label{}
\vert \nabla w \vert^{} & = & c^{}_{}\gamma^{} \vert x - \overline{x}_{0} \vert^{-\left( \gamma + 2 \right) } \vert x - \overline{x}_{0} \vert^{} \\ \nonumber
& = & c^{}_{} \gamma^{} \vert x - \overline{x}_{0} \vert^{-\left( \gamma + 1 \right) } \\ 
& \geq &c_{6} > 0, \ \ \ \ \mbox{in} \ \ D.
\end{eqnarray*}
Also we have $\langle \nu_{x}, e_{n}\rangle \geq c$ in $\left\lbrace v_{t_{0}} \leq 0 \right\rbrace \cap D$ (for $\varepsilon$ small enough). In fact, if $\varepsilon$ is small enough 
\begin{eqnarray*}
\label{}
\left\lbrace v_{t_{0}} \leq 0 \right\rbrace \cap D \subset \left\lbrace p \leq \frac{c_{0}\varepsilon}{2} \right\rbrace = \left\lbrace x_{n} \leq \frac{c_{0}\varepsilon}{2} - \sigma \right\rbrace \subset \left\lbrace x_{n} < 1/20 \right\rbrace.
\end{eqnarray*} 
We therefore conclude that 
\begin{eqnarray*}
\label{}
\langle \nu_{x}, e_{n}\rangle & = & \nonumber \frac{1}{\vert \overline{x}_{0} - x \vert}\langle \overline{x}_{0} - x, e_{n} \rangle \\ \nonumber
& \geq & \frac{5}{4}\langle \overline{x}_{0} - x, e_{n} \rangle  \\ \nonumber
& = & \frac{5}{4}\left(  \frac{1}{10} - r_{2} - x_{n} + \frac{1}{20} - \frac{1}{20}\right) \\ \nonumber
& > & c_{7}, \ \ \ \ \mbox{in} \ \ \left\lbrace v_{t_{0}} \leq 0 \right\rbrace \cap D.
\end{eqnarray*}
From (\ref{t 5.2.18}) and (\ref{t 5.2.19}) we obtain
\begin{eqnarray*}
\vert \nabla v_{t_{0}} \vert^{2} & \geq & \vert \partial_{n} v_{t_{0}}\vert^{2} \\  \nonumber 
& = & 1 + 2\tilde{c} \varepsilon^{}  + \tilde{c}  \varepsilon^{2}\vert \nabla w \vert^{2} \\  \nonumber 
& \geq & 1 + 2c_{9}\varepsilon  + c_{10}\varepsilon^{2}  \\  \nonumber 
&  \geq & 1 + \varepsilon^{2}.
\end{eqnarray*}
Hence
\begin{eqnarray*}
\vert \nabla v_{t_{0}} \vert^{2}  \geq 1 + \varepsilon^{2} > Q^2 \ \ \mbox{in} \ \ D \cap \mathfrak{F}\left( v_{t_{0}} \right).
\end{eqnarray*}
in $\left\lbrace v_{t_{0}} \leq 0 \right\rbrace \cap D$.
In particular, we have
\begin{eqnarray*}
\vert \nabla v_{t_{0}} \vert  > Q \ \ \mbox{in} \ \ D \cap \mathfrak{F}\left( v_{t_{0}} \right).
\end{eqnarray*}
Thus, $v_{t_{0}}$ is a strict subsolution in $D$ and by Lemma \ref{l 5.1} ($u$ is a viscosity solution of problem (\ref{P 5.1. introduc}) in $B_{1}\left( 0 \right)$) we conclude that $y_{0} \in B_{r_{3}}\left( \overline{x}_{0} \right)$. This is a contradiction. In fact, we would get
\begin{eqnarray*}
u\left( y_{0} \right) = v_{t_{0}}\left( y_{0} \right) = v\left( y_{0} \right) + t_{0} \leq p\left( y_{0} \right) + t_{0} < p\left( y_{0} \right) + \dfrac{c_{0}\varepsilon}{2}.
\end{eqnarray*}
which drives us to a contradiction to (\ref{NEW 5}). The Lemma \ref{Harnack lemma 5.2} is concluded.

\item[Case 2]: $\vert \nabla u (x_{0}) \vert \geq \frac{1}{4}$.\\

Since $u$ is $C^{1, \alpha}$ in $B_{\frac{1}{40}}(x_{0})$, there exists a constant $r_{0}=r_{0}(n, \mu) > 0$ such that  
\begin{eqnarray}
\label{NEW 2.1}
\vert \nabla u \vert \geq \dfrac{1}{8} \quad \text{in} \ B_{r_{0}}(x_{0}).
\end{eqnarray}
Then, $u$ satisfies 
\begin{eqnarray}
\label{NEW 2.2}
 \Delta u =  g  \quad \text{in} \ B_{r_{0}}(x_{0}),
\end{eqnarray}
where $g = \dfrac{f}{ \vert \nabla u \vert}$ with $\Vert g \Vert_{\infty} \leq \varepsilon^{2} 8^{\mu}$. Thus, by classical Harnack Inequality we obtain
\begin{eqnarray*}
\label{}
u\left( x \right) - p\left( x \right)  & \geq & c_{0}(u\left( x_{0} \right) - p\left( x_{0} \right)) - C \Vert f \Vert_{\infty} \\
& \geq & \frac{c_{0}\varepsilon}{2} - C_{1}\varepsilon^{2} \\
& \geq & c_{1}\varepsilon,
\end{eqnarray*}
for all $x \in B_{1/40}(x_{0})$, if $\varepsilon >0$ is sufficiently small. Now, we consider the barrie
\begin{eqnarray*}
\label{}
w\left( x \right) = \left \{ 
\begin{array}{lll}
c \left [ \vert x - x_{0} \vert^{-\gamma} - \left( \frac{4}{5}\right)^{-\gamma} \right ] , & \text{in} & B_{\frac{4}{5}}\left( x_{0} \right) \setminus B_{1/40}\left( x_{0} \right), \\
1, & \text{on} &  B_{1/40}\left( x_{0} \right), 
\end{array}
\right.  
\end{eqnarray*}
and the Lemma \ref{Harnack lemma 5.2} follows as in Case 1.
\end{proof}

\bigskip 

 Now we establish the main tool in the proof of Theorem \ref{th 5.1. introd.}.

\begin{theorem}
\label{t 5.2.1} 
Let $u$ be a viscosity solution to (\ref{P 5.1. introduc}) in $\Omega$ under assumptions (\ref{t 5.2.b})--(\ref{t 5.2.d}). There exists a universal constant $\tilde{\varepsilon} > 0$ such that, if $u$ satisfies at some $x_{0} \in \Omega^{+}\left( u \right) \cup F\left( u \right)$, 
\begin{eqnarray}
\label{t 5.2.2}
\left( x_{n} + a _{0}\right)^{+} \leq u \left( x \right) \leq \left( x_{n} + d_{0} \right)^{+} \ \ \mbox{in} \ \ B_{r}\left( x_{0} \right) \subset \Omega,
\end{eqnarray} 
with $\vert a_{0} \vert < \frac{1}{20}$ and
\begin{eqnarray}
\label{t 5.2.3}
d_{0} - a_{0} \leq \varepsilon r, \ \ \ \varepsilon \leq \tilde{\varepsilon} 
\end{eqnarray}
then
\begin{eqnarray}
\label{t 5.2.4}
\left( x_{n} + a _{1}\right)^{+} \leq u \left( x \right) \leq \left( x_{n} + d_{1} \right)^{+} \ \ \mbox{in} \ \ B_{\frac{r}{40}}\left( x_{0} \right)
\end{eqnarray}
with
\begin{eqnarray}
\label{t 5.2.5}
a_{0} \leq a_{1} \leq d_{1} \leq d_{0}, \ \ \ d_{1} - a_{1} \leq \left( 1 - c \right)\varepsilon r, 
\end{eqnarray} 
and $0 < c < 1$ universal.
\end{theorem}

\begin{proof}
With no loss of generality, we can assume  $x_{0}=0$ and $r=1$. We put $p\left( x \right) = x_{n} + a_{0}$ and by (\ref{t 5.2.2}) 
\begin{eqnarray}
\label{t 5.2.6}
p^{+}\left( x \right) \leq u \left( x \right) \leq \left( p\left( x \right) + \varepsilon \right)^{+} \ \  \left( d_{0} \leq a_{0} + \varepsilon \right). 
\end{eqnarray}
Then, since 
$$u \left( \frac{1}{10}e_{n} \right) \geq \left( p\left( \frac{1}{10}e_{n} \right) + \frac{\varepsilon}{2} \right)^{+} \quad \text{or} \quad u \left( \frac{1}{10}e_{n} \right) < \left( p\left( \frac{1}{10}e_{n} \right) + \frac{\varepsilon}{2} \right)^{+}$$ we can apply Lemma \ref{Harnack lemma 5.2} to obtain the result. 
\vspace{0,5 cm}\\
\end{proof}

\bigskip

From Harnack  inequality, Theorem \ref{t 5.2.1}, precisely as in \cite{DeSilva}, we obtain the following key estimate for flatness improvement.
\begin{corollary}
\label{c 5.1}
Let $u$ be a viscosity solution to (\ref{P 5.1. introduc}) in $\Omega$ under assumptions (\ref{t 5.2.b})--(\ref{t 5.2.d}). If $u$ satisfies (\ref{t 5.2.2}) then in $B_{1}\left( x_{0} \right)$ the function $\tilde{u}_{\varepsilon}:= \frac{u - x_{n}}{\varepsilon}$ has a H\"older modulus of continuity at $X_{0}$ outside of ball of radius $\varepsilon / \tilde{\varepsilon}$, i.e. for all $x \in \left( \Omega^{+}\left( u \right) \cup \mathfrak{F}\left( u \right) \right)\cap B_{1}\left( x_{0} \right)$ with $\vert x - x_{0} \vert \geq \varepsilon / \tilde{\varepsilon}$ 
\begin{eqnarray*}
\label{}
\vert \tilde{u}_{\varepsilon}\left( x \right) - \tilde{u}_{\varepsilon}\left( x_{0} \right)\vert \leq C \vert x - x_{0} \vert^{\gamma}. 
\end{eqnarray*}    
\end{corollary}


\section{Improvement of Flatness} \label{section4}

In this section we prove the \textit{improvement of flatness} lemma, from which the proof of main Theorem \ref{th 5.1. introd.} will follow via an interactive argument. Next we present the basic induction step towards $C^{1,\gamma}$ regularity at $0$.


\begin{theorem}[Improvement of flatness]\label{lemma2N}
Let  $u \in C(\Omega)$ be a viscosity solution to

\begin{eqnarray}
\label{(4.0.1)}
	\left \{ 
		\begin{array}{lll}
			\mathcal{L}_{\mu} u  =  f, & \text{ in } & \Omega_{+}\left( u \right),\\
			|\nabla u| = Q, & \text{ on } & \mathfrak{F}(u). 
		\end{array}
	\right.
\end{eqnarray}
whith  ($0 < \epsilon <1$)
\begin{equation}\label{(4.0.2N)}
\max \, \left\{ \|f\|_{L^{\infty}(\Omega)} , \,\,  \|Q-1\|_{L^{\infty}(\Omega)} \right\} \le \epsilon^{2}.
\end{equation}
Suppose that $u$ satisfies
\begin{equation}\label{(4.1N)}
    (x_n - \epsilon)^{+} \le u(x) \le (x_n + \epsilon)^{+} \quad \textrm{for}\,\, x \in B_1
\end{equation}
with $0 \in \mathfrak{F}(u)$. If $0 <r \le r_0$ for $r_0$ a universal constant and $0 < \epsilon \le \epsilon_0$ for some $\epsilon_{0}$ depending on $r$, then
\begin{equation}\label{(4.2)}
    \left(\langle x , \nu \rangle - r \frac{\epsilon}{2}\right)^{+} \le u(x) \le \left(\langle x , \nu \rangle + r \frac{\epsilon}{2}\right)^{+} \quad x \in B_r,
\end{equation}
with $|\nu| =1$, and $|\nu-e_n| \le C \epsilon^2$ for a universal constant $C>0$.
\end{theorem}

\begin{proof}
We divide the proof of this Lemma into $3$ steps. We use the following notation:
$$
    \Omega_{\rho}(u) := (B^{+}_{1}(u) \cup \mathfrak{F}(u)) \cap B_{\rho}.
$$

{\bf Step 1 - Compactness Lemma:}\, Fix $r \le r_0$ with $r_0$ universal (the precise $r_0$ will be given in Step 3). Assume by contradiction that we can find a sequence $\epsilon_{k} \to 0$ and a sequence  $\{u_{k}\}_{k \ge 1} \subset C(\Omega)$ be a sequence of viscosity solution to
\begin{equation}\label{seq}
\left\{
\begin{array}{rcl}
\mathcal{L}_{\mu}u_k &=& f_{k}  \quad \quad \quad \mbox{in} \,\,\,\,\, \Omega^{+}_{1}(u_{k})\\
|\nabla u_{k}|&=& Q_{k}(x)  \quad \,\, \mbox{on} \quad \mathfrak{F}(u_{k}) 
\end{array}
\right.
\end{equation}
with 
\begin{equation}\label{condk}
\max \left\{ \|f_{k}\|_{L^{\infty}}, \,\, \|Q_{k}-1\|_{L^{\infty}} \right\} \le \epsilon^{2}_{k},
\end{equation}
as $k \to \infty$, such that
\begin{equation}\label{(4.3)}
    (x_n - \epsilon_{k})^{+} \le u_{k}(x) \le (x_n + \epsilon_{k})^{+} \quad \textrm{for} \,\, x \in B_{1}, \,\, 0 \in \mathfrak{F}(u_k)
\end{equation}
but it does not satisfy the conclusion \eqref{(4.2)} of the Lemma. Let $v_{k} : \Omega_{1}(u_k) \rightarrow \mathds{R}$ defined by
$$
    v_{k}(x) := \frac{u_{k}(x)-x_n}{\epsilon_{k}}.
$$
Then \eqref{(4.3)} gives,
\begin{equation}\label{(4.4)}
    -1 \le v_{k}(x) \le 1 \quad \textrm{for} \,\, x \in \Omega_{1}(u_k).
\end{equation}
From Corollary \ref{c 5.1}, it follows that the function $v_k$ satisfies
\begin{equation}\label{(4.5)}
    |v_{k}(x)-v_{k}(y)| \le C |x-y|^{\gamma},
\end{equation}
for $C$ universal and
$$
    |x-y| \ge \epsilon_{k} / \bar \epsilon, \,\,\, x,y \in \Omega_{1/2}(u_k).
$$
From \eqref{(4.3)} it clearly follows that $\mathfrak{F}(u_k) \to B_1 \cap \{x_n=0\}$ in the Hausdorff distance. This fact and \eqref{(4.5)} together with Ascoli-Arzela give that as $\epsilon_k \to 0$ the graphs of the $v_k$ over $\Omega_{1/2}(u_k)$ converge(up to a subsequence) in the Hausdorff distance to the graph of a H\"older continuous function $u_{\infty}$ over $B_{1/2} \cap \{x_n \ge 0\}$. \\ {\bf Step 2 - Limiting Solution:} We claim that $\tilde{u}$ is a solution of the problem
\begin{eqnarray}
\label{limiting theor}
\left \{
\begin{array}{lll}
\Delta u_{\infty} = 0 \ \ \mbox{in} \ \ B^+_{\frac{1}{2}} \\
\partial_{n} u_{\infty} = 0 \ \ \mbox{on} \ \ \Upsilon_{1/2} \\
\end{array}
\right.
\end{eqnarray}
in viscosity sense. In fact, given a quadratic polynomial $P \left( x \right)$ touching $\tilde{u}$ at $x_{0} \in B_{\frac{1}{2}}\left( 0 \right)  \cap \left\lbrace  x_{n} \geq 0 \right\rbrace$ strictly by below we need to prove that
\begin{enumerate}
	\item[(i)] If $x_{0} \in \ B_{\frac{1}{2}}\left( 0 \right) \cap \left\lbrace  x_{n} > 0 \right\rbrace$ then $\Delta P \leq 0$;
	\item[(ii)] If $x_{0} \in \ B_{\frac{1}{2}}\left( 0 \right) \cap \left\lbrace  x_{n} = 0 \right\rbrace$ then $\partial_{n} P\left( x_{0} \right) \leq 0$.
\end{enumerate}

As in \cite{DeSilva}, there exist points $x_{j} \in \Omega_{\frac{1}{2}}\left( u_{j}\right)$, $x_{j} \rightarrow x_{0}$, and constants $c_{j} \rightarrow 0$ such that
\begin{eqnarray*}
u_{j}\left( x_{j}\right) = \tilde{P}\left( x_{j} \right) 
\end{eqnarray*}
and
\begin{eqnarray*}
u_{j}\left( x\right)\geq \tilde{P}\left( x \right) \ \ \mbox{in a neighborhood of} \ x_{j}   
\end{eqnarray*}
where 
\begin{eqnarray*}
\tilde{P}\left( x \right) = \varepsilon_{j}\left( P\left( x \right) + c_{j}  \right) + x_{n}.   
\end{eqnarray*} 
We have two possibilities:\\

\noindent (a) If $x_{0} \in B_{\frac{1}{2}} \cap \left\lbrace  x_{n} > 0 \right\rbrace$ then, since $P$ touches $u_{j}$ by below at $x_{j}$, we estimate

\begin{eqnarray*} 
\varepsilon_{j}^{2} & \geq&  f_{j}\left( x_{j}\right)\\ \nonumber
& \geq & \mathcal{L}_{\mu} \tilde{P} \\ \nonumber
& =  & \varepsilon^{}_{j} \vert \nabla \tilde{P} \vert^{\mu} \Delta \tilde{P} .
\end{eqnarray*}
Using that $\nabla \tilde{P} = \varepsilon_{j} \nabla P + e_{n}$ and taking $\varepsilon_{j} \longrightarrow 0$ we obtain
\begin{eqnarray*} 
\Delta P \leq 0.
\end{eqnarray*}

\noindent (b)  If $x_{0} \in \ B_{\frac{1}{2}} \cap \left\lbrace  x_{n} = 0 \right\rbrace$ we can assume, see \cite {DeSilva}, that 
\begin{eqnarray}
\label{strict subharmonic}
\Delta P > 0 
\end{eqnarray}
Notice that for $j$ sufficiently large we have $x_{j} \in \mathfrak{F}\left( u_{j}\right)$. In fact, suppose by contradiction that there exists a subsequence $x_{j_{n}} \in B^{+}_{1}\left( u_{j_{n}}\right)$ such that $x_{j_{n}} \rightarrow x_{0}$. Then arguing as in (i) we obtain
\begin{eqnarray*} 
\Delta P \leq  C\varepsilon_{j}, 
\end{eqnarray*}
which contradicts (\ref{strict subharmonic}) as $j_{n} \rightarrow \infty$. Therefore, there exists $j_{0} \in \mathbb{N}$ such that $x_{j} \in \mathfrak{F}\left( u_{j}\right)$ for $j \geq j_{0}$. 
Moreover, 
\begin{eqnarray*} 
\vert \nabla \tilde{P} \vert \geq 1 - \varepsilon_{j}\vert \nabla P\vert > 0, 
\end{eqnarray*}
for $j$ sufficiently large (we can assume that $j \geq j_{0}$). Since that $\tilde{P}^{+}$ touches $u_{j}$ by below we have
\begin{eqnarray*} 
\vert \nabla \tilde{P} \vert ^{2} \leq Q_{j}\left( x_{j}\right) \leq \left(1 + \varepsilon^{2}_{j} \right)   . 
\end{eqnarray*} 
Then, we obtain
\begin{eqnarray*} 
\vert \nabla \tilde{P} \vert ^{2} \leq \left(1 + \varepsilon^{2}_{j} \right). 
\end{eqnarray*} 
Moreover,
\begin{eqnarray*}
\vert \nabla \tilde{P} \vert ^{2} & = &  \varepsilon^{2}_{j}\vert \nabla P \left( x_{j}\right) \vert^{2} + 1 + 2\varepsilon_{j}\partial_{n}P \left( x_{j}\right),
\end{eqnarray*}
where we have used  $\vert \nabla \tilde {P} \vert^{2} \leq C$.
In conclusion, we obtain
\begin{eqnarray}
\label{t 5.3.4} 
\varepsilon^{2}_{j}\vert \nabla P \left( x_{j}\right) \vert^{2} + 1 + 2\varepsilon_{j}\partial_{n}P \left( x_{j}\right) \leq 1 + \varepsilon^{2}_{j}. 
\end{eqnarray}
Hence, dividing (\ref{t 5.3.4}) by $\varepsilon_{j}$ and taking $j \rightarrow \infty$ we obtain $\partial_{n}P\left( x_{0}\right) \leq 0$.   \\ 

The choice of $r_{0}$ and the conclusion of the Theorem \ref{th 5.1. introd.} follows from the regularity of $\tilde{u}$:\\
{\bf Step 3 - Improvement of flatness:}\, From the previous step, $u_{\infty}$ solve \eqref{limiting theor} and from \eqref{(4.4)},
$$
    -1 \le u_{\infty} \le 1 \quad \textrm{in} \,\, B_{1/2} \cap \{x_n \ge 0\}.
$$
From Lemma \ref{reg_Laplace} and the bound above we obtain that, for the given $r$,
$$
    |u_{\infty}(x) - u_{\infty}(0) - \langle \nabla u_{\infty}(0) , x \rangle| \le C_0 r^2 \quad \textrm{in} \,\, B_r \cap \{x_n \ge 0\},
$$
for a universal constant $C_0$. In particular, since $0 \in \mathfrak{F}(u_{\infty})$ and $\frac{\partial u_{\infty}(0)}{\partial \mu} =0$, we estimate
$$
   \langle \tilde{x} , \tilde{\nu} \rangle - C_1 r^2 \le u_{\infty}(x) \le \langle \tilde{x} , \tilde{\nu}\rangle + C_0 r^2 \quad \textrm{in} \,\, B_r \cap \{x_n \ge 0\},
$$
where $\tilde{\nu}_{i} = \langle \nabla u_{\infty}(0) , e_i \rangle$, $i=1, \ldots, n-1$, $|\tilde{\nu}| \le \tilde{C}$ and $\tilde{C}$ is a universal constant. Therefore, for $k$ large enough we get,
$$
    \langle \tilde{x} , \tilde{\nu} \rangle - C_1 r^2 \le v_{k}(x) \le \langle \tilde{x} , \tilde{\nu} \rangle + C_1 r^2 \quad \textrm{in} \,\, \Omega_{r}(u_k).
$$
From the definition of $v_{k}$ the inequality above reads
\begin{equation}\label{(4.9)}
    \epsilon_{k} \tilde{x} \cdot \tilde{\nu} + x_n - \epsilon_{k} C_1 r^2 \le u_k \le  \epsilon_{k} \langle \tilde{x} , \tilde{\nu} \rangle + x_n + \epsilon_{k} C_1 r^2 \quad \textrm{in} \,\, \Omega_{r}(u_k).
\end{equation}
Define 
$$
    \nu := \frac{1}{\sqrt{1+ \epsilon^{2}_{k}}}(\epsilon_{k} \tilde{\nu}, 1).
$$
Since, for $k$ large,
$$
 1 \le \sqrt{1+\epsilon^{2}_{k}} \le 1+ \frac{\epsilon^{2}_{k}}{2},
$$
we conclude from \eqref{(4.9)} that
$$
\langle x , \nu \rangle - \frac{\epsilon^{2}_{k}}{2} r - C_1 r^2 \epsilon_{k} \le u_{k} \le \langle x , \nu \rangle + \frac{\epsilon^{2}_{k}}{2} r + C_1 r^2 \epsilon_k \quad \textrm{in}\,\,\, \Omega_{r}(u_k).
$$
In particular, if $r_0$ is such that $C_1r_0 \le \frac{1}{4}$ and also $k$ is large enough so that $\epsilon_{k} \le \frac{1}{2}$ we find
$$
    \langle x , \nu \rangle - \frac{\epsilon_{k}}{2} r \le u_k \le \langle x , \nu \rangle + \frac{\epsilon_{k}}{2} r \quad \textrm{in} \,\, \Omega_{r}(u_k),
$$
which together with \eqref{(4.3)} implies that
$$
    \left(\langle x , \nu \rangle - \frac{\epsilon_{k}}{2} r\right)^{+} \le u_k \le  \left(\langle x , \nu \rangle + \frac{\epsilon_{k}}{2} r\right)^{+} \quad \textrm{in}\,\,\, B_r.
$$
Thus the $u_k$ satisfy the conclusion of the Lemma, and we reached a contradiction.
\end{proof}


\section{Regularity of the free boundary}\label{reg_fron_livre}
In this section we will prove the Theorem \ref{th 5.1. introd.} and via a blow-up from Theorem \ref{th 5.1. introd.} we will present the proof of Theorem \ref{Holder1}. The proof of Theorem \ref{th 5.1. introd.} is based on flatness improvement coming from Harnack type estimates and it follows closely the work of \cite{DeSilva}. Hereafter, we will assume 
\begin{equation} \label{H}
	|Q(x)-Q(y)| \le \tau(|x-y |) \quad \textrm{for} \,\,\, x,y \in B_1,
\end{equation}
where the modulus of continuity $\tau$ satisfies
 \begin{equation}\label{modulo4}
 	\tau(t) \lesssim Ct^{\beta},
 \end{equation}
 for some $0 < \beta <1$ and $C>0$.

 
\begin{proof}[Proof of Theorem \ref{th 5.1. introd.}]
 The idea of proof is to iterate the Theorem \ref{lemma2N} in the appropriate geometric scaling. Let $u$ be a viscosity solution to the free boundary problem 
\begin{eqnarray}
 \label{defviscsol}
	\left \{ 
		\begin{array}{lll}
			\mathfrak{L}_{\mu} u  =  f, & \text{ in } & B^+_1(u),\\
			|\nabla u| = Q, & \text{ on } & \mathfrak{F}(u). 
		\end{array}
	\right.
\end{eqnarray}
where $B^+_1(u) = \{x \in B^+_1 : u(x) >0\}$ and $ \mathfrak{F}^+(u) := \partial B^{+}_{1}(u) \cap B_1$.  Let us fix $\bar r >0$ to be a universal constant such that
\begin{equation} \label{smallRN}
	\overline{r}^{\beta} \le \min \left\{ \left(\frac{1}{2}\right)^{2},  r_0 \right\},
\end{equation}
with $r_0$ the universal constant in Theorem \ref{lemma2N}.  For the chose $\overline{r}$, let $\epsilon_{0} := \epsilon_{0}(\overline{r})$ give by Theorem \ref{lemma2N}.  Now, let 
\begin{equation}\label{choiceN}
	\overline{\epsilon}:= \epsilon^{2}_{0} \quad \textrm{and} \quad \epsilon = \epsilon_{k} := 2^{-k} \epsilon_0.
\end{equation}
Our choice of $\overline{\epsilon}$ guarantees that  
\begin{equation} \label{Gu}
	(x_n -  \epsilon_0)^{+} \le u(x) \le (x_n +  \epsilon_0)^{+} \quad \textrm{in} \,\,\, B_1.
\end{equation}
Thus by Theorem \ref{lemma2N} 
$$
	\left(\langle x , \nu_{1} \rangle - \bar r \frac{\epsilon_0}{2}\right)^{+} \le u(x) \le \left(\langle x , \nu_{1}\rangle + \bar r \frac{\epsilon_0}{2}\right)^{+} \quad \textrm{in} \,\,\, B_{\bar r},
$$
with $|\nu_{1}|=1$ and $|\nu_{1}-\nu_{0}| \le C{ \epsilon_0}^{2}$ (where $\nu_{0}=e_n$). \\
{\bf Smallness regime:}  Consider the sequence of rescalings $u_{k}: B_1 \rightarrow \mathds{R}$
$$
    u_{k}(x):= \frac{u(\lambda_{k}x)}{\lambda_{k}}
$$
with $\lambda_{k} = \overline{r}^{k}$, $k=0,1,2,\ldots$, for a fixed $\overline{r}$  as in \eqref{smallRN}.
Then each $u_{k}$ satisfies in the following free boundary problem
\begin{eqnarray}
 \label{}
	\left \{ 
		\begin{array}{lll}
			\mathfrak{L}_{\mu}  u_{k}  =  f_{k}, & \text{ in } & B^+_1(u_{k}),\\
			|\nabla u_{k}| = Q_{k}, & \text{ on } & \mathfrak{F}(u_{k}). 
		\end{array}
	\right.
\end{eqnarray}
$$
     f_{k}(x):= \lambda_{k}f(\lambda_{k}x) \quad \textrm{and} \quad Q_{k}(x) := Q(\lambda_k x).
$$
 We claim that for the choices made in \eqref{choiceN} the assumption \eqref{(4.0.2N)} are holds. Indeed, in $B_1$
\begin{eqnarray*}
  |f_{k}(x)| &\le& \|f\|_{L^{\infty}} \lambda_k \le \overline{\epsilon} {\bar r}^k  \le \epsilon^{2}_0 2^{-2k} =(\epsilon_0 2^{-k} )^{2} = \epsilon_{k}^{2},\\
|Q_k(x) -1| &=& |Q(\lambda_k x) - Q_k(0)| \le \tau(1) \lambda^{\beta}_{k}  \le \overline{\epsilon} {\bar r}^ {k \beta}  \le (\epsilon_0 2^{-k})^{2} = \epsilon_{k}^{2}
\end{eqnarray*}
Therefore, we can iterate the argument above and obtain that 
\begin{equation}\label{interativo}
    ( \langle x , \nu_{k} \rangle -\epsilon_{k})^{+} \le u_{k}(x) \le (\langle x , \nu_{k} \rangle + \epsilon_{k})^{+} \quad \textrm{in} \,\,\, B_1,
\end{equation}
with $|\nu_{k}|=1$, $|\nu_{k}-\nu_{k+1}| \le C \epsilon_{k}$ ($\nu_{0}=e_n$), where $C$ is a universal constant. Thus, we have
\begin{equation} \label{int2}
	\left( \langle x , \nu_k \rangle - \frac{\epsilon_0}{2^k} \overline{r}^{k}\right)^{+} \le u(x) \le \left(\langle  x , \nu_k \rangle  + \frac{\epsilon_{0}}{2^{k}} \overline{r}^{k}\right)^{+} \quad \textrm{in} \quad B_{\overline{r}^{k}}
\end{equation}
with 
\begin{equation} \label{int3}
	|\nu_{k+1} - \nu_{k}| \le C \frac{\epsilon_0}{2^k}.
\end{equation}
Which \eqref{int2} implies that
\begin{equation} \label{int5}
	\partial \{u>0\} \cap B_{\overline{r}^k} \subset \left\{|\langle x, \nu_k \rangle| \le \frac{\epsilon_0}{2^k} \overline{r}^{k}\right\}
\end{equation}
This implies that $B_{3/4} \cap \mathfrak{F}(u) $ is a $C^{1,\gamma}$ graph. In fact, by \eqref{int3} we have that $\{\nu_k\}_{k \ge 1}$ is a Cauchy sequence, therefore the limit
$$
	\nu(0) := \lim_{k \to \infty} \nu_k
$$
exists. Yet from \eqref{int3} we conclude
$$
	|\nu_k - \nu(0)| \le C \frac{\epsilon_0}{2^k}.
$$
From \eqref{int5} we have
\begin{equation} \label{int4}
	|\langle x , \nu_{k}\rangle| \le \frac{\epsilon_0}{2^k} \overline{r}^k.
\end{equation}
Fix $x \in B_{3/4} \cap \partial \{u >0\}$ and choose $k$  such that
$$
	\overline{r}^{k+1} \le |x| \le \overline{r}^k.
$$
Then
\begin{eqnarray*}
	|\langle x , \nu(0) \rangle| &\le& |\langle x, \nu(0) - \nu_k \rangle| + |\langle x , \nu_k\rangle| \\
	&\le& |\nu(0) - \nu_k| |x| + \frac{\epsilon_0}{2^k} \overline{r}^k\\
	&\le& C\frac{\epsilon_0}{2^k} |x| + \frac{\epsilon_0}{2^k} \overline{r}^k\\
	&\le& C\frac{\epsilon_0}{2^k} (|x| + \overline{r}^k) \\
	&\le& C \frac{\epsilon_0}{2^k} (|x| + \frac{\overline{r}^{k+1}}{\overline{r}})\\
	&\le& C \frac{\epsilon_0}{2^k} (1 + \frac{1}{\overline{r}}) |x|.
	\end{eqnarray*}
	From the convenient choice of $k$, we have  $|x| \ge \overline{r}^{k+1}$. Hence, if we define $0 < \gamma <1$ such that
	$$
		\frac{1}{2}^{} = \overline{r}^{\gamma}
	$$
	i.e, define  $\gamma := \frac{\ln(2)}{\ln(\overline{r}^{-1})}$. Thus, we have
	\begin{eqnarray*}
		|\langle x , \nu(0)\rangle| &\le& C (\frac{1}{2})^k (1 + \overline{r}^{-1}) |x| \\
		&=& C (\frac{1}{2})^{k+1} (1 + \overline{r}^{-1}) 2 |x| \\
		&\le& C (1+\overline{r}^{-1}) \epsilon_0 |x|^{1+\gamma} \le C \epsilon_0 |x|^{1+\gamma}.
	\end{eqnarray*}
Finally, we obtain
$$
	\partial\{u >0\} \cap B_{\overline{r}^{k}} \subset \left\{ \langle x, \nu(0) \rangle \le C \epsilon_0 \overline{r}^{k(1+ \gamma)}\right\},
$$
which implies that $\partial \{u>0\}$ is a differentiable surface at $0$ with normal $\nu(0)$.  Applying this argument at all points in $\partial\{u>0\} \cap B_{3/4}$ we see that $\partial \{u>0\} \cap B_{3/4}$ is in fact a $C^{1,\gamma}$ surface.
\end{proof}


The next lemma proof of a standard result that is Lipschitz continuity and non-degeneracy of a solution $u$ to 
\begin{eqnarray}
\label{NN}
	\left \{ 
		\begin{array}{lll}
			\mathfrak{L}_{\mu} u  =  f, & \text{ in } & \Omega_{+}\left( u \right),\\
			|\nabla u| = Q, & \text{ on } & \mathfrak{F}(u). 
		\end{array}
	\right.
\end{eqnarray}

\begin{lemma}\label{nao-deg}
Let $u \in C(\Omega)$ be a viscosity solution to \eqref{NN}. Given $\epsilon \in (0,1)$, we can find a universal constant $\tilde{\epsilon}$ such tha if $\epsilon \in (0,  \tilde{\epsilon}]$, $\mathfrak{F}(u) \cap B_1 \not= \varnothing$, $\mathfrak{F}(u)$ is a Lipschitz graph in $B_2$ and
\begin{equation} \label{flu}
 \max \left\{ \|f\|_{L^{\infty}(\Omega)},\,\, \|Q-1\|_{L^{\infty}(\Omega)} \right\} \le \epsilon^{2},
\end{equation}
then $u$ is Lipschitz and non-degenerate in $B^{+}_{1}(u)$ i.e. there exists universal conconstants $c_0,c_1 >0$
$$
    c_0 \textrm{dist}(z,\mathfrak{F}(u)) \le u(z) \le c_1 \textrm{dist}(z,\mathfrak{F}(u)) \quad \textrm{for all } \,\, z \in B^{+}_{1}(u).
$$
\end{lemma}

\begin{lemma}[Compactness] \label{compactness}
Let $u_k$ be a sequence of (Lipschitz) viscosity solutions to
$$
\left\{
\begin{array}{rcl}
\mathfrak{L}_{\mu} u_k &=& f_k \quad  \mbox{in} \quad \Omega^{+}(u_k) ,\\
|\nabla u_k | &=&Q_k \quad \mbox{on} \quad \mathfrak{F}(u_k)
\end{array}
\right.
$$
where $f_k$ and $Q_k$ satisfies the assumption \eqref{flu}. Assume that
\begin{enumerate}
\item[(i)]
$
    u_{k} \to u_{\infty} \quad \text{uniformly on compacts};
$
\item[(ii)]
$
    \partial \{u_{k} >0\} \to \partial \{u_{\infty} >0\} \quad \textrm{locally in the Hausdorff distance;}
$
\item[(v)] $\|f_k\|_{L^{\infty}} + \|Q_k-1\|_{L^{\infty}} = o(1)$, as $k \to \infty$
\end{enumerate}
Then $u_{\infty}$ be a viscosity solution of
$$
\left\{
\begin{array}{rcl}
\Delta  u_{\infty}  &=& 0,  \quad \mbox{in} \quad \Omega^{+}(u_{\infty}) ,\\
|\nabla u_{\infty}| &=& 1,  \quad \mbox{on} \quad \mathfrak{F}(u_{\infty}),
\end{array}
\right.
$$
 in the viscosity sense.
\end{lemma}
\begin{proof}
The proof that follow the same scheme of the model Lemma \ref{lemma2N} (see also \cite{DeSilva} Lemma 7.3). 
\end{proof}

Although not strictly necessary, we use the following Liouville type result for  global viscosity solutions to a one-phase homogeneous free boundary problem, that could be of independent interest. The result is more general, but we will only show the result for a one-phase problems..
\begin{lemma}\label{Liouville}
Let $v: \mathds{R}^{n} \rightarrow \mathds{R}$ be a non-negative viscosity solution to
$$
\left\{
\begin{array}{rcl}
 \Delta v_{} &=& 0,  \quad \mbox{in} \quad \{v>0\} ,\\
\langle \nabla v , \nu \rangle &=&1,  \quad \mbox{on} \quad \mathfrak{F}(v):= \partial \{v>0\}.
\end{array}
\right.
$$
 Assume that $\mathfrak{F}(v) = \{x_n =g(x'), x' \in \mathds{R}^{n-1}\}$ with $\textrm{Lip}(g) \le M$. Then $g$ is linear and
$$
    v(x)= x_{n}^{+} .
$$
\end{lemma}
\begin{proof}
Let's follow the ideas of \cite{DeSilva}. Initially, assume for simplicity, $0 \in \mathfrak{F}(v)$. Also, balls (of radius $\rho$ center at $0$) in $\mathds{R}^{n-1}$ are denote by $B^{'}_{\rho}$. By the regularity theory in \cite{C1}, since $v$ is a solution in $B_2$, the free boundary $\mathfrak{F}(v)$ is $C^{1,\gamma}$ in $B_1$ with a bound depending only on $n$ and on $M$. Thus,
$$
    |g(x')-g(0)-\nabla g(0) \cdot x'| \le C |x'|^{1+\gamma} \quad \textrm{for} \,\,\,x^{'}\in B^{'}_{1}
$$
with $C$ depending only on $n$, $M$. Moreover, since $v$ us a global solution, the rescaling
$$
    g_{\lambda}(x'):= \frac{1}{\lambda}g(\lambda x'), \,\, x' \in B^{'}_{2}
$$
which preserves the same Lipschitz constant as $g$, satisfies the same inequality as above, i.e.
$$
    |g_{\lambda}(x')-g_{\lambda}(0)-\nabla g_{\lambda}(0) \cdot x'| \le C |x'|^{1+\gamma}\quad \textrm{for} \,\,\,x^{'}\in B^{'}_{1}.
$$
Thus,
$$
    |g(y')-g(0)-\nabla g(0)\cdot y'| \le C \frac{1}{R^{\gamma}} |y'|^{1+\gamma}, \,\,\, y' \in B^{'}_{R}.
$$
Passing to the limit as $R \to \infty$  we obtain the desired claim.
\end{proof}


Finally we can prove Theorem \ref{Holder1}. In this section we finally the proof of our second main theorem.

\begin{proof}[ Proof of Theorem \ref{Holder1}]
Let $\overline{\epsilon}>0$ be the universal constant in Theorem  \ref{th 5.1. introd.} and $u$. Without loss of generality, assume $Q(0)=1$. Consider the re-scaled function
$$
    u_k := u_{\delta_{k}}(x) = \frac{u(\delta_{k}x)}{\delta_k},
$$
with $\delta_{k} \to 0$ as $k \to \infty$. Each $u_k$ solves
$$
\left\{
\begin{array}{rcl}
\mathfrak{L}_{\mu} u_k&=& f_k \quad  \mbox{in} \quad B^{+}_{1}(u_k),\\
| \nabla u_{k} |&=&Q_k \quad \mbox{on} \quad \mathfrak{F}(u_k),
\end{array}
\right.
$$
with
$$
    f_{k}(x):= \delta_{k}f(\delta_{k}x) \quad \textrm{and} \quad Q_{k}(x) := Q(\delta_k x).
$$
Furthermore, for $k$ large, the assumption \eqref{flu} are satisfied for the universal constant $\bar \epsilon$. In fact, in $B_1$ we have
\begin{eqnarray*}
   |f_{k}(x)| &=& \delta_{k} |f(\delta_{k}x)| \le \delta_{k} \|f\|_{L^{\infty}} \le \overline{\epsilon}^{2}\\
    |Q_{k}(x)-1| &=& |Q_{k}(x)-Q_{k}(0)| \le  \tau(1) \delta^{\beta}_{k} \le \overline{\epsilon}^{2}
\end{eqnarray*}
for $k$ large enough. Therefore, using non-degeneracy (see Lemma \ref{nao-deg}) and uniform Lipschitz continuity of the $u_{k}'s $ (see Lemma \eqref{nao-deg} ), standard arguments imply that (up to a subsequence)
\begin{enumerate}
\item[(i)] There exists $u_{\infty} \in C(\Omega)$ such that
$
    u_{k} \to u_{\infty} \quad \textrm{uniformly on compacts};
$
\item[(ii)]
$
    \partial \{u_{k} >0\} \to \partial \{u_{\infty} >0\} \quad \textrm{locally in the Hausdorff distance;}
$
\item[(iii)] $\|f_k\|_{L^{\infty}} + \|Q_k-1\|_{L^{\infty}} = o(1)$, as $k \to \infty$
\end{enumerate}
and, as in Lemma \ref{compactness}, the blow-up limit $u_{\infty}$ solves the global homogeneous one-phase free boundary problem
 $$
\left\{
\begin{array}{rcl}
\Delta u_{\infty} &=& 0,  \quad \mbox{in} \quad  \{u_{\infty}>0\},\\
|\nabla u_{\infty}|&=&1,  \quad \mbox{on} \quad \mathfrak{F}(u_{\infty}).
\end{array}
\right.
$$
Since $\mathfrak{F}(u)$ is a Lipschitz graph in a neighborhood of $0$ we also have from have $(i)-(iii)$ that  $\mathfrak{F}(u_{\infty})$ is Lipschitz continuous. Thus, follows the Lemma \ref{Liouville}  that $u_{\infty}$ is a so-called one-phase solution, i.e. (up to rotations)
$$
    u_{\infty} = x^{+}_{n}.
$$
Thus, for $k$ large enough we have
$$
    \|u_{k}-u_{\infty}\|_{L^{\infty}} \le \overline{\epsilon}
$$
and the facts thar $u_{k}$ is $\overline{\epsilon}$-flat say in $B_1$ i.e
$$
    (x_n-\overline{\epsilon})^{+} \le u_{k}(x) \le (x_n+\overline{\epsilon})^{+}, \,\,\, x \in B_1.
$$
Therefore, we can apply our flatness Theorem \ref{lemma2N} and conclude that $\mathfrak{F}(u_k)$ and hence $\mathfrak{F}(u)$ is $C^{1,\gamma}$, for some $\gamma \in (0,1)$.
\end{proof}

\subsection*{Acknowledgments}

\hspace{0.65cm} RAL and GCR thanks the Analysis research group of UFC for fostering a pleasant and productive scientific atmosphere.


\end{document}